\documentclass{mystyle}

\title[Classifying Isometries]{Classifying Isometries}

\author{Lillian MacArthur}
\author{Honglin Zhu}

\begin{document}

\begin{abstract}
An isometry is a geometric transformation that preserves distances between pairs of points. We present methods to classify isometries in the Euclidean plane, and extend these methods to spherical, single elliptical, and hyperbolic geometry. We then classify all isometries of the plane equipped with the $\ell_p$ metric for any $p \geq 1$ and $p = \infty$.
\end{abstract}

\maketitle

\section{Introduction}
Transformational geometry is an area that explores mappings of geometric spaces to themselves, and the properties that can be broken or preserved under these maps. One of the most important types of these maps is an isometry, in which the distance between any two points is preserved. Other mappings such as affine transformations, conformal mappings, and contractions are also featured in transformational geometry. For more discussions on transformational geometry, see Chapter~$5$ in \cite{sibley2015}. 

Classifying the isometries of a space is obviously a useful venture, but in order to give an isometry of a certain space, we must first specify what distance means. For example, when we talk about distances in the plane, we usually mean the Euclidean distance, where the distance between points $(x_1, y_1)$ and $(x_2, y_2)$ is given by $(x_2-x_1)^2+(y_2-y_1)^2$. However, there are many other possible distance functions that the plane can be equipped with. To formalize the properties of a distance function that a space can have, we define metric spaces. 

\begin{definition}
    A \textit{metric space} $(X, d)$ is a set $X$ equipped with a distance function $d$ which satisfies the following conditions:
    \begin{itemize}
        \item $d(x,x) = 0$;
        \item  $d(x,y) = d(y,x)$ (Symmetry);
        \item  $d(x,y) \ge 0$ (Positivity);
        \item  $d(x,y) + d(y,z) \ge d(x,z)$ (Triangle Inequality).
    \end{itemize}
\end{definition}

The first three conditions are intuitive to our understanding of distances. The fourth one, however, comes from the fact the distance between two points should always be shorter than a path that takes a detour to a third point. This can be seen in that for any triangle $ABC$, the distance from $A$ to $B$ must be at most the distance from $A$ to $C$ added to the distance from $C$ to $B$, hence the name triangle inequality.

\begin{example}
    One metric other than the Euclidean metric on the real plane is the taxicab metric. In taxicab geometry, the distance between any two points is given by the sum of their $x$- and $y$-displacements: 
    \[
    d_T((x_1,y_1), (x_2, y_2)) = |x_2-x_1| + |y_2-y_1|.
    \]
    The first three conditions for the distance function are trivially satisfied. The triangle inequality holds because 
    \begin{align*}
        d_T((x_1, y_1),(x_2, y_2)) + d_T((x_2, y_2),(x_3, y_3)) 
        &= |x_2-x_1| + |y_2-y_1| + |x_3-x_2| + |y_3-y_2| \\
        &\ge |x_3-x_1| + |y_3-y_1| \\
        &= d_T((x_1, y_1),(x_3, y_3)).
    \end{align*}
    Thus, the plane equipped with the taxicab metric is indeed a metric space. 
\end{example}

Taxicab geometry is a great example of a metric space that uses a completely different understanding of distance than Euclidean geometry has. Its name comes from the idea of a taxicab in a large city where all the streets run in a grid, so it can only get around by travelling parallel to the $x$- or $y$-axis. Then, the distance the taxicab must travel is equal to the $x$-displacement plus the $y$-displacement. Taxicab geometry can be very useful in urban planning.

With the definition of a metric space, we are now ready to give a formal definition of an isometry.
\begin{definition}
    An \textit{isometry} of a metric space $(X, d)$ is a bijective map $\varphi: X \to X$ such that the distance function is preserved: for all $x, y \in X$, $d(x, y) = d(\varphi(x), \varphi(y))$.
\end{definition}

\begin{example}\label{ex_taxi_trans}
    A translation on $\mathbb{R}^n$ displaces every point by the same vector: $\varphi(x) = x + v$. This is clearly an isometry under the Euclidean metric as well as the taxicab metric. 
\end{example}

From \Cref{ex_taxi_trans} we can see that some isometries can exist in multiple different metric spaces. This suggests a connection between all the different types of isometries that we can explore.

However, to explore isometries, it would be helpful to figure out a method to represent transformations beyond defining them individually. One helpful method for this is to model transformations with matrices.

\begin{definition}
    An \textit{affine transformation} is a map $F: \mathbb{R}^n \to \mathbb{R}^n$ such that $F(x) = Mx + b$, where $M$ is some invertible $n \times n$ matrix, and $b \in \mathbb{R}^n$. 
\end{definition}

In \Cref{sec_two}, we classify isometries in Euclidean space as compositions of reflections. In \Cref{sec_three}, we move to spherical, single elliptical, and hyperbolic metric spaces, and adapt the method we used for Euclidean space. Next, in \Cref{sec_four}, we explore the isometries in Taxicab geometry further. Finally, in \Cref{sec_five}, using a weaker version of the Mazur-Ulam theorem, we classify all isometries for the plane equipped with the $\ell^p$ metric, for all $p > 1$ and $p = \infty$.

\section{Isometries in Euclidean Space}\label{sec_two}
Euclidean space is one of the most common and relevant metric spaces to examine in geometry. Euclid, its namesake, investigated geometry using only a straightedge and a compass, and built all of his theorems off of those tools. In Euclidean space, distance is measured the way we intuitively see distance, using the Pythagorean theorem-based distance formula we learn in school.

\subsection{Isometries in the Euclidean plane}
In this subsection, following \cite{sibley2015}, we give a classification of the isometries in the Euclidean plane.

\begin{lemma}\label{lem_three_points}
    Let $A, B, C \in \mathbb{R}^2$ be distinct non-collinear points. For any two points $P, P' \in \mathbb{R}^2$, if the respective distances from $P$ and $P'$ to $A, B, C$ are all equal, then $P = P'$.
\end{lemma}
\begin{proof}
Without loss of generality, suppose $P$ is distinct from $A$ and $B$. The circles centered at $A$ and $B$ going through $P$ intersect at $P$ and exactly one other point, which we denote by $Q$. Suppose for the sake of contradiction that $P' \neq P$. Then $P' = Q$ and $C$ is equidistant to $P$ and $Q$. This implies that $C$ lies on the perpendicular bisector of $P$ and $Q$. However, $A$ and $B$ both lie on this line, which leads to a contradiction, as $A, B, C$ are assumed to be non-collinear.
\end{proof}

\Cref{lem_three_points} implies that an isometry is determined by its action on three distinct non-collinear points: since if we know the images of three such points, then the image of any other point is the unique point with the appropriate distances to the images of the three points. This will become vital for classifying isometries.

In fact, this proof lets us classify all isometries in the Euclidean plane as the composition of reflections.
\begin{definition}\label{def_reflection}
    A \textit{reflection} in the plane across line $\ell$ is a map that takes each $x \in \mathbb{R}^2$ to the unique point $r(x) \in \mathbb{R}^2$ such that $\ell$ is the perpendicular bisector of the segment with endpoints $x$ and $r(x)$. 
\end{definition}

A reflection can be understood as ``flipping'' the plane across a certain line (see \Cref{fig_reflection}).

\begin{figure}[ht]
      \centering 
      \begin{tikzpicture}[scale=1]
\draw[step=1cm,gray,very thin] (-3.9,-3.9) grid (3.9,3.9);
\draw[thick,->] (-4,0) -- (4,0) node[anchor=south] {$x$};
\draw[thick,->] (0,-4) -- (0,4) node[anchor=west] {$y$};
\draw (-2,-4) -- (2,4);
\draw[thick,->] (0.5,2.25) -- (1.5,1.75);
\draw[thick,->] (0.5,-2.75) -- (-2.5,-1.25);
\end{tikzpicture}
\caption{A Reflection}\label{fig_reflection}
\end{figure}
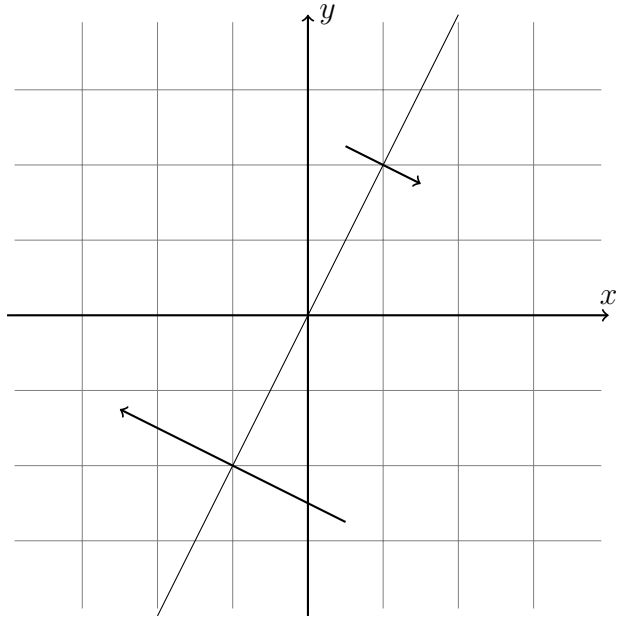

It is clear from definition that a reflection is an isometry in the Euclidean plane. With this definition, we are now ready to classify the isometries in the Euclidean plane in terms of reflections. 
\begin{theorem}\label{thm_euclidean_isometry}
    Any isometry in the Euclidean plane is the composition of at most three reflections.
\end{theorem}
\begin{proof}
Suppose $\varphi: \mathbb{R}^2 \to \mathbb{R}^2$ is a Euclidean isometry. Let $A, B, C$ be distinct non-collinear points with images $\varphi(A), \varphi(B), \varphi(C)$. By \Cref{lem_three_points}, it suffices to show that $A, B, C$ can be mapped to their respective images by at most three mirror reflections. 

Define $\mu_A$ as the reflection along the perpendicular bisector of $A$ and $\varphi(A)$, which maps $A$ to $\varphi(A)$. Suppose $\mu_A(B) = B'$, $\mu_A(C) = C'$. If $B'= B$ and $C'=C$, we are finished. Otherwise, without loss of generality we can let $B' \ne \varphi(B)$. Define $\mu_B$ as the reflection across the perpendicular bisector of $B'$ and $\varphi(B)$, which similarly maps $B'$ to $\varphi(B)$. Because $\mu_B$ is a reflection which is an isometry, the distance from $A$ to $B$ must equal the distance from $\varphi(A)$ to $B'$, and because the total transformation is an isometry, the distance from $A$ to $B$ must equal the distance from $\varphi(A)$ to $\varphi(B)$, so all three distances are equivalent and $\varphi(A)$ is equidistant from $B'$ and $\varphi(B)$. Therefore, $\varphi(A)$ must lie on the perpendicular bisector of $B'$ and $\varphi(B)$, or the line of reflection of $\mu_B$, and remain fixed under this reflection. 

For point $C'$, however, this is not necessarily true, so we denote the point $\mu_B(C')$ as $C''$. If $C''$=$\varphi(C)$, we are finished, and the isometry can be expressed as $\mu_B \circ \mu_A$. Otherwise, define $\mu_C$ to be the reflection across the perpendicular bisector of $C''$ and $\varphi(C)$, which maps $C''$ to $\varphi(C)$. By similar reasoning as above, $\varphi(A)$ and $\varphi(B)$ are held fixed. Therefore, this isometry can be represented as $\mu_C \circ \mu_B \circ \mu_A$.
\end{proof}

This is an invaluable result; it gives us a complete classification of isometries in the Euclidean plane by expressing it as a composition of at most three reflections:
\begin{itemize}
    \item A single reflection will produce a reflection.
    \item 2 reflections with intersecting lines of reflection will produce a rotation around the intersection point.
    \item 2 reflections with parallel lines of reflection will produce a translation (that translates with a distance twice as long as that between the lines of reflection and perpendicular to them).
    \item 3 reflections with parallel lines of reflection will produce another reflection.
    \item All other combinations of 3 reflections are a \textit{glide reflection}, or a composition of translation and a reflection. 
\end{itemize}

Rotations and translations are \textit{direct isometries} which preserve orientation, while reflections and glide reflections are \textit{indirect isometries} and do not. In this system of classification, we can see that any isometry composed of an odd number of mirror reflections is indirect, and any isometry composed of an even number is direct. This intuitively makes sense, as every reflection flips orientation; an even number puts it right again, while an odd number leaves it flipped.

\subsection{Isometries in higher dimensional Euclidean space}
Metric spaces don't have to be flat; they can be in three dimensional or even higher dimensional space. Though spaces of four dimensions or higher are hard for us to visualize, they can have important implications in math, science, and computer programming, so they are important to explore. Our conclusion above can be easily generalized to any $n$-dimensional space $\mathbb{R}^n$, as long as we remember that in $\mathbb{R}^n$, reflections are across a $(n-1)$-dimensional hyperplane, not a line. In three dimensions, this means they're across a plane. 

\begin{theorem}\label{thm_euc_ndim_isometry}
    Any isometry in $n$-dimensional Euclidean space is the composition of at most $n+1$ reflections.
\end{theorem}

The proof \Cref{thm_euc_ndim_isometry} follows the same argument as in that of \Cref{thm_euclidean_isometry}. An analog of \Cref{lem_three_points} can be established with the fact that the intersection of two $(n-1)$-dimensional spheres is either the empty set, a single point, or an $(n-2)$-dimensional sphere.

\section{Isometries in Spherical, Single Elliptical, and Hyperbolic Space}\label{sec_three}
While Euclidean space is integral to our understanding of our world and geometry, it is not the only system of geometry we can explore. Isometries can exist in any metric space; and we can change around the type of space and the type of distance function used to change the type of isometries that we can find. 

One way of doing this is by examining the parallel postulate. This was one of the five rules that Euclid came up with while beginning to study geometry, and states that for any line and any point, there exists only one line that intersects the point and is parallel to the line. Unlike some other of his postulates, like the existence of points and lines, this feels much less intuitive. 

Mathematicians, attempting to prove the parallel postulate, discovered that it cannot be proved; and when we take it away, a whole new world of geometries opens up. In these geometries, a line can have many or no lines parallel to it that all intersect a single point. We will explore three types of these geometries in this paper:
\begin{enumerate}
    \item Spherical geometry is the type of geometry we observe on the surface of a sphere. We define a line as a circle that wraps around the sphere and lies on a plane that intersects the center (a "great circle"). Then, every line has a finite length, and intersects every other line at two points. Spherical geometry has \textit{positive curvature}; space ``curves up'' so that all lines will intersect eventually.
    \item Single elliptical geometry is similar to spherical geometry, except antipodal points are considered equivalent. It can also be envisioned as a circle where lines that exist one edge come in from the opposite side. Yet another way of seeing is the set of all lines through the unit sphere. No matter how its represented, it is a fascinating geometry to explore, because it lacks the \textit{separation axiom}: a line does not divide the plane in two, and a line segment can extend on either side of a line without intersecting it. It, like spherical geometry, has positive curvature.
    \item Hyperbolic geometry is a space with negative curvature, where space ``curves away" so that a line can have infinite lines parallel to it that all intersect the same point. Parallel lines in fact have two types: sensed parallel, where the lines are at the smallest angle they can be at without intersecting, and ultraparallel lines which are not. Each line has a sensed right, sensed left, and infinite ultraparallel lines through any point. Hyperbolic geometry can be represented in many ways. One can imagine looking onto a hyperboloid from its top, so it appears to be a circle but distance steadily increases as a line moves towards the outside so it will never reach (see \Cref{fig_hyperbolic} from \cite{Tamfang}).
    \begin{figure}[ht]
        \centering
        \includegraphics[width = 0.3 \textwidth]{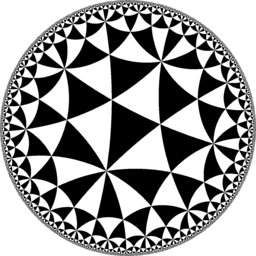}
        \caption{This drawing demonstrates the properties of hyperbolic geometry, as each line curves away from the rest}
        \label{fig_hyperbolic}
    \end{figure}
\end{enumerate}
For more discussions on non-Euclidean geometry, see Chapter~4 in \cite{sibley2015}.

Isometries can have many different forms in these geometries. As lines and the expansion of space are so different in them, the isometries can be different as well as the same as Euclidean. In order to classify them, though, we can still use the same method of reflections.

\begin{theorem}
    Isometries in spherical, single elliptical, and hyperbolic spaces can be described as the composition of 3 reflections.
\end{theorem}

The proof of \Cref{lem_three_points} primarily relies on 3 statements; that an isometry can be defined with 3 points, that every two points have a unique perpendicular bisector, and that the composition of reflections will produce the isometry. These can all be proved to remain true in these three geometries, but we will not dive into the proofs of these assertions in this paper.   

However, the nature of these lines can be very different, depending on the geometry. Chapter~4 of \cite{sibley2015} names the isometries in these three geometries. 

In spherical geometry, all lines intersect at two points.
\begin{itemize}
    \item One reflection makes a reflection.
    \item Two reflections make a rotation around their two intersection points.
    \item Three reflections make a glide reflection.
\end{itemize}

In single elliptical geometry, all lines intersect at a single point.
\begin{itemize}
    \item One reflection makes a reflection.
    \item Two reflections or three reflections with a common intersection point make a rotation around the intersection point.
    \item Three reflections with different intersection points make a rotation composed with a reflection. 
\end{itemize}

In hyperbolic geometry, lines can intersect, be sensed parallel, or ultraparallel.
\begin{itemize}
    \item One reflection is a reflection.
    \item Two reflections with intersecting lines of reflection make a rotation.
    \item Two reflections with sensed parallel lines of reflection make a \textit{horolation}, a strange hyperbolic rotation around a point at infinity. 
    \item Two reflections with ultraparallel lines of reflection make a \textit{hyperbolic translation}, another isometry that becomes more strange in hyperbolic space, in which repeated iterations move points around in \textit{horocycles} rather than in lines.
    \item Three reflections make a reflection or a \textit{hyperbolic glide reflection}, 
\end{itemize}

These isometries follow the same general structure as Euclidean space, with translations, rotations, and reflections being composed in different ways, but their differences let whole new categories of isometries arise or eliminate ones that worked fine in Euclidean space.

\section{Isometries in Taxicab Geometry}\label{sec_four}
We introduced taxicab geometry in the introduction as an example of a different metric we can equip the real plane with. Unlike in spherical, single elliptical, or hyperbolic geometry, lines behave in exactly the same way; only distance is measured differently. Distance turns out to be just as essential to a metric space as the space is, though, as changing the distance function changes how circles are expressed, how points exist in relation to one another, and most importantly to us, whether a transformation is an isometry. 

Recall that the distance between two points $(x_1, y_1)$ and $(x_2, y_2)$ in taxicab geometry is given by $d_T((x_1, y_1), (x_2, y_2)) = |x_2-x_1| + |y_2-y_1|$. 

By defining distance as the sum of the $x$- and $y$-displacement, we get a very unintuitive but fascinating system of geometry. Lines and points remain the same, as they are not defined in terms of distance, but conic sections and other distance-related objects are very different, such as circles ellipses, and hyperbolas (see \Cref{fig_taxicab_unit_circle}, \Cref{fig_taxicab_ellipse}, and \Cref{fig_taxicab_hyperbola}).

\begin{figure}[ht]
      \centering 
      \begin{tikzpicture}[scale = 2]
    \draw[thick] (0,1) -- (1,0) -- (0,-1) -- (-1,0) -- cycle;
    \draw[thick,->] (-1.5,0) -- (1.5,0) node[anchor=north west] {$x$};
    \draw[thick,->] (0,-1.5) -- (0,1.5) node[anchor=south east] {$y$};
    \foreach \x in {-1,1}
    \draw (\x cm,1pt) -- (\x cm,-1pt) node[anchor=north] {$\x$};
    \foreach \y in {-1,1}
    \draw (1pt,\y cm) -- (-1pt,\y cm) node[anchor=east] {$\y$};
    \node at (-0.2,-0.2) {$0$};
\end{tikzpicture}
\caption{The unit circle in taxicab geometry.}\label{fig_taxicab_unit_circle}
\end{figure}
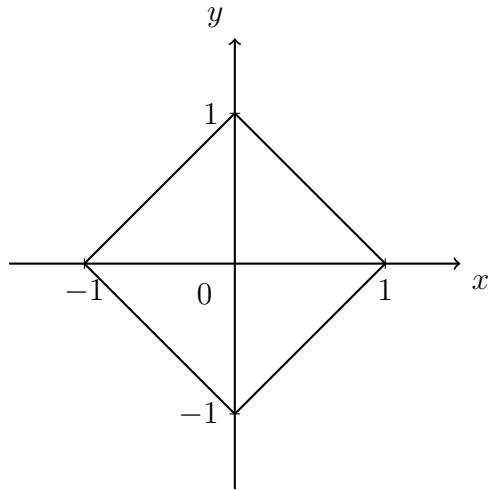

\begin{figure}[ht]
\centering
\begin{minipage}{0.45\textwidth}
\begin{tikzpicture}[scale = 1.5]
    \draw[thick] (0,-1) -- (1,-1) -- (2,0) -- (2,2) -- (1,3) -- (0,3) -- (-1,2) -- (-1,0) -- cycle;
    \draw[thick,->] (-1.5,0) -- (2.5,0) node[anchor=north west] {$x$};
    \draw[thick,->] (0,-1.5) -- (0,3.5) node[anchor=south east] {$y$};
    \foreach \x in {-1,1,2}
    \draw (\x cm,1pt) -- (\x cm,-1pt) node[anchor=north] {$\x$};
    \foreach \y in {-1,1,2,3}
    \draw (1pt,\y cm) -- (-1pt,\y cm) node[anchor=east] {$\y$};
    \node at (-0.2,-0.2) {$0$};
    \draw[fill] (0, 0) circle (0.05cm) node[anchor=south east] {$F_1$};
    \draw[fill] (1, 2) circle (0.05cm) node[anchor=south east] {$F_2$};
\end{tikzpicture}
\caption{An ellipse with foci at $(0, 0), (1, 2)$ and sum of distances to foci $5$.}\label{fig_taxicab_ellipse}
\end{minipage}
\begin{minipage}{0.45\textwidth}
\begin{tikzpicture}[scale = 1.5]
    \draw[thick] (0.25,3.5) -- (0.25, 2) -- (1, 1.25) -- (2.5, 1.25);
    \draw[thick] (-1.5,0.75) -- (0, 0.75) -- (0.75, 0) -- (0.75, -1.5);
    \draw[thick,->] (-1.5,0) -- (2.5,0) node[anchor=north west] {$x$};
    \draw[thick,->] (0,-1.5) -- (0,3.5) node[anchor=south east] {$y$};
    \foreach \x in {-1,1,2}
    \draw (\x cm,1pt) -- (\x cm,-1pt) node[anchor=north] {$\x$};
    \foreach \y in {-1,1,2,3}
    \draw (1pt,\y cm) -- (-1pt,\y cm) node[anchor=east] {$\y$};
    \node at (-0.2,-0.2) {$0$};
    \draw[fill] (0, 0) circle (0.05cm) node[anchor=south east] {$F_1$};
    \draw[fill] (1, 2) circle (0.05cm) node[anchor=south east] {$F_2$};
\end{tikzpicture}
\caption{A hyperbola with foci at $(0, 0), (1, 2)$ and difference of distances to foci $1.5$.}\label{fig_taxicab_hyperbola}
\end{minipage}
\end{figure}

As can be seen, taxicab shapes are very different from our usual perceptions of them. Additionally, the strange shape of the circle means that we cannot use our reflection-based classification, because circles can intersect at infinitely many points. On the other hand, the less symmetric shape of the circle also puts more restrictions on the possible isometries. 


\begin{lemma}\label{lem_translation}
    If $\varphi: \mathbb{R}^2 \to \mathbb{R}^2$ is a taxicab isometry, then $\varphi'(x):= \varphi(x) - \varphi(0)$ is an isometry that fixes the origin. In other words, every taxicab isometry can be expressed as a taxicab isometry that fixes the origin composed with a translation.
\end{lemma}
\begin{proof}
    We can plug $0$ into this formula to get $\varphi'(0)= \varphi(0) - \varphi(0) = (0,0)$. Therefore, $\varphi'$ preserves the origin. It is also straightfoward to check $\varphi'$ is still an isometry. 
\end{proof}

Thus, it suffices to classify taxicab isometries that fix the origin. As the unit circle is the set of all points that are at distance $1$ from the origin, an isometry fixing the origin must map the unit circle back onto itself. 

It is easy to see that any two points in the unit circle have a distance of at most two between them; since all points have distance $1$ from the origin, and at the very worst a path between two points can be made going through the origin. Two antipodal points, on the opposite sides of the unit circle from each other, must be distance $2$ from each other. However, is there any set of points that are all distance $2$ away from each other with a size greater than two?

\begin{proposition}\label{prop_four_corners}
The four corners of the taxicab unit circle are the unique four points on the unit circle such that the distance between any two consecutive points going counterclockwise around the unit circle is $2$. 
\end{proposition}
\begin{proof}
Let $A = (x_1, y_1), B = (x_2, y_2), C = (x_3, y_3), D = (x_4, y_4)$ be any four points on the unit circle and assume they are ordered counterclockwise as such. Consider the axes-parallel rectangle $R$ that bounds these four points; namely, the rectangle bounded by the lines $y = \min y_i$, $y = \max y_i$, $x = \min x_i$, $x = \max x_i$. Observe that the perimeter of $R$ is at most $8$, with equality if and only if the four points are the four corners of the taxicab unit circle. 

We claim that $d_T(A, B) + d_T(B, C) + d_T(C, D) + d_T(D, A)$ is precisely the perimeter of $R$. To show this, it suffices to prove
\[ 
    |x_1 - x_2| + |x_2 - x_3| + |x_3 - x_4| + |x_4 - x_1| = 2 (\max x_i - \min x_i),
\]
since the corresponding statement for the $y_i$'s follow by symmetry. Since the points are arranged counterclockwise, the left hand side covers the segment between $\max x_i$ and $\min x_i$ twice, establishing the equality. 
\end{proof}

Since the pairwise distances between the four corners $(1,0), (0,1), (-1, 0), (0, -1)$ are all $2$, \Cref{prop_four_corners} implies that any isometry must permute these four points. This leaves only eight options for isometries that fix the origin.

\begin{theorem}\label{thm_taxicab_isometry}
    There are exactly eight taxicab isometries that fix the origin, namely the identity, the reflections across the $y$-axis, $x$-axis, the line $y=x$, the line $y=-x$, and the rotations by $\frac{\pi}{2}$, $\pi$, and $-\frac{\pi}{2}$.
\end{theorem}
\begin{proof}
    It is easy to check that these eight are indeed distinct taxicab isometries fixing the origin. 
    
    First we show that the only permutations that an isometry fixing the origin can induce on the four corners of the unit circle are those induced by the eight isometries in the theorem statement. By symmetry, it suffices to show that no isometry can fix $(-1, 0), (0, -1)$ and swap $(1,0), (0,1)$. Indeed, the point $(-0.5, 0.5)$ has distance $1$ to the points $(-1, 0)$ and $(0,1)$. However, the only point that is at distance $1$ to their images $(-1, 0)$ and $(1, 0)$ is the origin, which is not on the unit circle. 

    Let $\varphi'$ be an isometry fixing the origin. We compose it with one of the eight isometries listed in the theorem statement to obtain a new isometry $\varphi$ which fixes the four corners of the unit circle. For any point $(a, 1 - a)$ on the segment of the unit circle in the first quadrant, it has distances to $(1, 0)$ and $(0, 1)$ summing to two. No point on the unit circle outside this segment satisfies this. Also, it is the unique point on the segment with distance $2a$ to $(0, 1)$. By a similar argument for the other segments, we see that $\varphi$ must act as the identity on the unit circle.

    For any point $(x, 0)$ on the $x$-axis, its image must have distance $|1-x|$ to $(1, 0)$ and $|1+x|$ to $(-1, 0)$. It is easy to see that the image must be $(x, 0)$. Similarly, $\varphi$ acts as the identity on the $y$-axis. Thus, any circle centered at the origin has its four corners fixed by $\varphi$. By the same argument as above, this implies that the entire circle is fixed by $\varphi$. Therefore, $\varphi$ must be the identity. It follows that any isometry fixing the origin can only be one of the eight given in the problem statement.
\end{proof}
\begin{remark}
    The isometries of taxicab geometry that fix the origin form a group under composition that is isomorphic to the symmetry group of a square (in particular, the symmetry group of the taxicab unit circle). 
\end{remark}

\section{Isometries in $\ell^p$ Spaces}\label{sec_five}
The $\ell^p$ category of metric spaces contains taxicab, Euclidean, and other distance metrics for geometry. Investigating it lets us expand our knowledge of distance outside of the realistic and into the theoretical, letting us truly dive into the nature of isometries.

The plane $\mathbb{R}^2$ equipped with the $\ell^p$ metric is a metric space with the distance function given by 
\[
d_{p}((x_1, y_1), (x_2, y_2)) = (|x_2-x_1|^p+|y_2-y_1|^p)^{1/p},
\]
where $p \geq 1$. 

As $p$ varies, the nature of the metric space changes. When $p=1$, we find taxicab geometry. Euclidean geometry occurs when $p=2$. Other values of $p$ lack specific names, but we can prove wide conclusions about all well-defined values of $p$ and the isometries within all of them.

Unfortunately, for $p<1$, the distance metric is not well defined. Recall that every distance metric must obey the triangle inequality; that $d(X,Y) + d(Y,Z) \ge d(X,Z)$. However, take $p=0.5$, where the distance function would be 
\[
    (|x_2-x_1|^{0.5}+|y_2-y_1|^{0.5})^{1/0.5} = (\sqrt{|x_2-x_1|}+\sqrt{|y_2-y_1|})^2.
\]
This means that the distance from $(1,0)$ to $(0,1)$ is $4$, while the distance from each to the origin is $1$, making an impossible triangle. 

It is also possible to take $p = \infty$, where the distance is defined to be $\max \{|x_2-x_1|, |y_2-y_1|\}$. This is what the distance function approaches as $p$ goes to $\infty$. For the $\ell^{\infty}$ geometry, the unit circle is a square (see \Cref{fig_l_infinity_unit_circle}), and much like taxicab geometry which has a different square unit circle, the isometries include only the symmetries of the $\ell^{\infty}$ unit circle composed with a translation. 

The cases of the taxicab geometry and the $\ell^{\infty}$ suggest that we should use the unit circle to help us classify isometries. As can be seen in \Cref{fig_l_3_unit_circle}, the unit circle in $\ell^{p}$ essentially approaches the $\ell^{\infty}$ square as $p$ increases.

\begin{figure}[ht]
\centering
\begin{minipage}{0.45\textwidth}
\begin{tikzpicture}[scale = 2]
    \draw[thick] (1,1) -- (1,-1) -- (-1,-1) -- (-1,1) -- cycle;
    \draw[thick,->] (-1.5,0) -- (1.5,0) node[anchor=north west] {$x$};
    \draw[thick,->] (0,-1.5) -- (0,1.5) node[anchor=south east] {$y$};
    \foreach \x in {-1,1}
    \draw (\x cm,1pt) -- (\x cm,-1pt) node[anchor=north] {$\x$};
    \foreach \y in {-1,1}
    \draw (1pt,\y cm) -- (-1pt,\y cm) node[anchor=east] {$\y$};
    \node at (-0.2,-0.2) {$0$};
\end{tikzpicture}
\caption{The unit circle under the $\ell^{\infty}$ metric.}\label{fig_l_infinity_unit_circle}
\end{minipage}
\begin{minipage}{0.45\textwidth}
    \begin{tikzpicture}[scale = 2]
    \draw[thick, domain = 0:90] plot ({cos(\x)^(2/3)}, {sin(\x)^(2/3)});
    \draw[thick, domain = 0:90] plot ({-cos(\x)^(2/3)}, {sin(\x)^(2/3)});
    \draw[thick, domain = 0:90] plot ({cos(\x)^(2/3)}, {-sin(\x)^(2/3)});
    \draw[thick, domain = 0:90] plot ({-cos(\x)^(2/3)}, {-sin(\x)^(2/3)});
    \draw[thick,->] (-1.5,0) -- (1.5,0) node[anchor=north west] {$x$};
    \draw[thick,->] (0,-1.5) -- (0,1.5) node[anchor=south east] {$y$};
    \foreach \x in {-1,1}
    \draw (\x cm,1pt) -- (\x cm,-1pt) node[anchor=north] {$\x$};
    \foreach \y in {-1,1}
    \draw (1pt,\y cm) -- (-1pt,\y cm) node[anchor=east] {$\y$};
    \node at (-0.2,-0.2) {$0$};
\end{tikzpicture}
\caption{The unit circle under the $\ell^{3}$ metric.}\label{fig_l_3_unit_circle}
\end{minipage}
\end{figure}

To fully classify $\ell^p$ isometries, we must introduce the Mazur-Ulam theorem, a result proven by Stanisław Mazur and Stanisław Ulam in $1936$ which states that a surjective isometry between two normed spaces over $\mathbb{R}$ must be affine. A proof of this result is beyond the scope of this paper and can be found in \cite{jussi2003}, but we will only need a weaker version of this theorem, which assumes the spaces to be strictly convex (See \Cref{strict_convex_def}). 

\begin{definition}
    A real vector space $V$ is a \textit{normed space} if it is equipped with a norm $||\cdot||: V \to \mathbb{R}$ such that for all vectors $x,y$ and scalars $\alpha$:
    \begin{itemize}
        \item $||x|| \ge 0$;
        \item If $||x|| = 0$, then $x$ is the zero vector;
        \item $||\alpha x|| = |\alpha| \cdot ||x||$;
        \item $||x|| + ||y|| \ge ||x+y||$. (This is the Triangle Inequality in a different form).
    \end{itemize}
\end{definition}
Intuitively, the definition of a norm is just a way to measure the length of vectors. It can be checked that all the $\ell^p$ spaces are normed spaces.
\begin{remark}
    All normed spaces are automatically metric spaces, with distance function induced by the norm: $d(X,Y) := ||Y-X||$.
\end{remark}
\begin{definition}\label{strict_convex_def}
    A normed vector space $(X, || \cdot ||)$ is \textit{strictly convex} if the closed unit ball $B(0, 1) := \{x \in X| ||x|| \leq 1\}$ is strictly convex. That is, if the segment joining two distinct points on the boundary of the unit ball meets the boundary in only these two points. 
\end{definition}
We see that the $\ell^p$ spaces are strictly convex for all $1 < p < \infty$, so to deal with these cases, we only need the following weaker version of the Mazur-Ulam theorem.
\begin{theorem}\label{thm_mazur_ulam}
    An isometry of a strictly convex normed space is an affine transformation. 
\end{theorem}
\begin{proof}
    Let $\varphi: X \to X$ be an isometry. Since an isometry is continuous, to show it is affine it suffices to prove $\varphi(\lambda x + (1 - \lambda) y) = \lambda \varphi(x) + (1 - \lambda) \varphi(y)$ for all $x, y \in X$, $\lambda \in [0,1]$. This claim is a standard exercise in real analysis and is left to the reader to verify. 

    Fix any distinct $x, y \in X$ and let $d = d(x, y) = d(\varphi(x), \varphi(y))$ be the distance between them. Denote the closed ball around a point $x$ with radius $r$ by $B(x, r)$. For any $\lambda \in [0,1]$, we claim that 
    \[
        B(x, \lambda d) \cap B(y, (1 - \lambda) d) = \{(1 - \lambda) x +  \lambda y\}.
    \]
    Indeed, $d(x, (1 - \lambda) x +  \lambda y) = ||\lambda (x - y)|| = \lambda d$ and the same computation for $y$ shows that the point is contained in the intersection of the two unit balls. On the other hand, let $z \in B(x, \lambda d) \cap B(y, (1 - \lambda) d)$. Then by the triangle inequality,
    \[
        d = d(x, y) \leq d(z, y) + d(z, x) \leq \lambda d + (1 - \lambda) d = d.
    \]
    The second inequality is an equality if and only if $z$ is on the boundary of the two closed balls. If $z \neq (1 - \lambda) x +  \lambda y$, then we have two distinct points on the boundary of two closed balls. By strict convexity, the interior of the segment between the two points must lie in the interior of the two closed balls, which is impossible since the two balls do not share any interior point as seen above. 

    Since $\varphi$ is a bijection, $\varphi((1 - \lambda) x +  \lambda y)$ is the unique point of 
    \begin{align*}
        \varphi(B(x, \lambda d) \cap B(y, (1 - \lambda) d)) 
        &= \varphi(B(x, \lambda d)) \cap \varphi(B(y, (1 - \lambda) d)) \\
        &= B(\varphi(x), \lambda d) \cap B(\varphi(y), (1 - \lambda) d).
    \end{align*}
    But clearly the right hand side contains $\lambda \varphi(x) + (1 - \lambda) \varphi(y)$, so we must have
    \[
        \varphi((1 - \lambda) x +  \lambda y) = \lambda \varphi(x) + (1 - \lambda) \varphi(y),
    \]
    as desired.
\end{proof}
Similar to \Cref{lem_translation}, any isometry in $\ell^p$ space is the composition of an isometry fixing the origin and a translation, so again it suffices to consider isometries fixing the origin.

With \Cref{thm_mazur_ulam}, we can now classify all isometries in all $\ell^p$ spaces. It turns out that we get the exact same set of isometries as in \Cref{thm_taxicab_isometry}. 
\begin{theorem}\label{thm_main}
    For any $p \geq 1$, $p \neq 2$, (including $p = \infty$), there are exactly eight isometries of $\mathbb{R}^2$ equipped with the $\ell^p$ metric that fix the origin. These are the identity, the reflections across the $y$-axis, $x$-axis, the line $y=x$, the line $y=-x$, and the rotations by $\frac{\pi}{2}$, $\pi$, and $-\frac{\pi}{2}$.
\end{theorem}
\begin{proof}
    The cases $p = 1$ and $p = \infty$ are already done by \Cref{thm_taxicab_isometry} ($p = \infty$ following essentially the same proof). 
    
    Now suppose $1 < p < \infty$ and $p \neq 2$. By \Cref{thm_mazur_ulam}, all isometries are affine. Let $\varphi: \mathbb{R}^2 \to \mathbb{R}^2$ be an isometry fixing the origin. Then $\varphi(x) = Mx$ for some nonsingular $2 \times 2$ matrix $M$. 

    Let $C = \{x \in \mathbb{R}^2| d(0, x) = 1\}$ be the unit circle and $D = \{x \in \mathbb{R}^2| d(0, x) \leq 1\}$ the unit disk. Since $\varphi$ is affine, it ``streches'' area by a constant factor, namely $|\det (M)|$. As $\varphi$ maps $D$ onto itself, it must preserve areas. Consider the triangle $T$ with vertices $(0, 0), (1, 0), (0, 1)$ as well as the sector $S$ of $D$ bounded by the segments from the origin to $(1, 0)$ and $(0, 1)$. Then $\varphi$ preserve the areas of both $T$ and $S$. Suppose $\varphi((1, 0)) = A$ and $\varphi((1, 0)) = B$. Then $S$ must be mapped to the sector bounded by the segments from the origin to $A$ and $B$. By the symmetry of $D$ (see \Cref{fig_l_3_unit_circle}), such a sector can have the same area as $S$ if and only if it the bounding segments have an angle of $\pi/2$. 

    On the other hand, $T$ is mapped to the right triangle with vertices $(0, 0)$, $A$, $B$. However, the Euclidean distance between the origin and any point on $C$ is either at most $1$ when $p < 2$ or at least $1$ when $p > 2$, with the only four points at Euclidean distance exactly $1$ being $(0, \pm 1), (\pm 1, 0)$. Thus, for the area of $T$ to be preserved, $\varphi$ must permute these four points. The only eight matrices $M$ that satisfies this correspond to the eight isometries listed in the theorem statement. 
\end{proof}

It is perhaps a little surprising that these different metrics all admit the same set of isometries. Even though different values of $p$ can give rise to extremely different geometries, they share isometries, and this connection elucidates much about their connections.

\section*{Acknowledgements}
This paper is the result of PRIMES Circle, in which the second author mentored the first. The authors would like to thank the MIT math department and the organizers of PRIMES for providing this opportunity. 

\printbibliography
\end{document}